\documentclass[11pt]{amsart}

\usepackage{amsthm, amsfonts, amssymb, amscd, rotating}
\usepackage[pagebackref,colorlinks]{hyperref}
\usepackage{tikz-cd}
\usepackage{geometry}
\usepackage{marginnote}
\usepackage{aligned-overset}
\usepackage[utf8]{inputenc}
\usepackage{xcolor}
\definecolor{darkgreen}{rgb}{0,0.5,0}
\usepackage{mathtools}

\theoremstyle{definition}
\newtheorem{ntn}{Notation}[section]
\newtheorem{dfn}[ntn]{Definition}
\theoremstyle{plain}
\newtheorem{lem}[ntn]{Lemma}
\newtheorem{prp}[ntn]{Proposition}
\newtheorem{thm}[ntn]{Theorem}
\newtheorem{introthm}{Theorem}
\newtheorem*{introprp}{Proposition E}
\newtheorem{cor}[ntn]{Corollary}

\theoremstyle{definition}

\newtheorem{rem}[ntn]{Remark}

\numberwithin{equation}{section}

\newcommand{\z}{\mathbb{Z}}

\newcommand{\F}{\mathbb{F}}

\newcommand{\m}{\mathfrak{m}}

\newcommand{\OO}{\mathcal{O}}

\newcommand{\Sl}{\mathfrak{sl}}

\newcommand{\ppp}{\mathfrak{p}}
\newcommand{\qqq}{\mathfrak{q}}

\newcommand{\arr}{\rightarrow}

\newcommand{\Gg}{\Gamma}
\newcommand{\GL}{{\rm GL}}
\newcommand{\PGL}{{\rm PGL}}

\newcommand{\SL}{{\rm SL}}
\newcommand{\GE}{{\rm GE}}
\newcommand{\Ee}{{\rm E}}

\renewcommand{\char}{{\rm char}}

\newcommand{\im}{{\rm im}}

\newcommand{\SK}{{\rm SK}}

\newcommand{\ab}{{\rm ab}}

\newcommand{\Kk}{{\rm K}}

\newcommand{\Spec}{{\rm Spec}}

\renewcommand{\aa}{{A^\times}}

\newtheoremstyle{athm}
{}
{}
{\itshape}
{}
{\scshape}
{}
{.5em}
{\thmnote{#3}}
\theoremstyle{athm}

\begin{document}

\title[On the abelianization of congruence subgroups of $\SL_2$ over $S$-integers]
{On the abelianization of congruence subgroups of $\SL_2$ over $S$-integers}
\author[P. H. Amorim]{Pedro H. Amorim}
\author[I. V. Picinini]{Isadora V. Picinini}
\author[B. R. Ramos]{Bruno R. Ramos}
\author[T. Verissimo]{Thiago Verissimo}
\address{\sf Instituto de Ci\^encias Matem\'aticas e de Computa\c{c}\~ao (ICMC), 
Universidade de S\~ao Paulo, S\~ao Carlos, São Paulo, Brasil}
\address{ \sf Department of Mathematics, Aarhus University, Ny Munkegade 118, 8000 Aarhus C, Denmark
}
\email{amorim.alves@usp.br}
\email{isadoravanzella@usp.br}
\email{brramos2050@gmail.br}
\email{thiagovlg@usp.br}

\begin{abstract}
In this work, we compute the first integral homology, or abelianization, of the congruence subgroups $\Gg(A, \m_A), \Gg_1(A, \m_A)$, and $\Gg_0(A, \m_A)$ for a local ring $A$ with maximal ideal $\m_A$, showing that $H_1(\Gg(A, \m_A), \z)$ is isomorphic to the additive group of $\mathfrak{sl}_2(\m_A/\m_A^2)$. We then use these results to determine the structure of the groups $H_1(\Gg(\mathcal{O}_{K, S}, \ppp), \z)$, $H_1(\Gg_1(\mathcal{O}_{K, S}, \ppp), \z)$ and $H_1(\Gg_0(\mathcal{O}_{K, S}, \ppp), \z)$, where $\OO_{K,S}$ is a Dedekind domain of arithmetic type, not totally imaginary, $|S| \geq 2$, and $\ppp$ is a nonzero prime ideal. The computations are given in terms of the residue field $\kappa(\ppp)$ and the known $H_1(\SL_2(\OO_{K, S}), \z)$. As a consequence, we also obtain the torsion subgroup of their second integral cohomology. These results will be of paramount importance for a forthcoming work concerning $H_2(\SL_2(\mathcal{O}_{K,S}), \z)$.\\


\end{abstract}

\maketitle

%%%%%%%%%%%%%%%%%%%%%%%%%%%%%%%%%%%%%%%%%%%%%%%%%%%%%%%%%%%%%%%%%%%%%%%%%%%%%%%%%%
\section*{Introduction} \label{int}

The power of (co)homology of arithmetic groups lies in its ability to encode information about spaces in a computable way. Consequently, it has seen an utmost impact in various branches of mathematics, finding applications in algebraic $K$-theory, hyperbolic geometry, and arithmetic geometry. However, describing the homology groups explicitly is generally a tough task, and has been a central subject of recent research articles (see, for example, \cite{h2016}, \cite{BBT2025-2}, \cite{BBT2025-1}, and \cite{B-E2024}). 

Computing these groups is of particular significance for number theory, as seen by the study of reciprocity laws via Hecke operators \cite{venkatesh2019} and the theory of automorphic forms \cite{venkatesh2017}. Moreover, in an upcoming work \cite{PIBT}, we establish valuable connections between $H_2(\SL_2(\OO_{K,S}), \z)$ and the special value $\zeta_K(-1)$ of the Dedekind zeta function associated to the field of fractions $K$. For this, we rely on explicit computations on the homology of congruence subgroups, done in this article.

This justifies the main goal of this article: to investigate the first integral homology of
\[\Gg(A, \ppp) = \bigg\{\begin{pmatrix}
        a & b \\ c & d
    \end{pmatrix} \in \SL_2(A) : a-1,\ d-1,\ b,\ c \in \ppp \bigg\},\]
\[\Gg_1(A, \ppp) = \bigg\{\begin{pmatrix}
        a & b \\ c & d
    \end{pmatrix} \in \SL_2(A) : a-1,\ d-1,\ c \in \ppp \bigg\},\]
and
\[\Gg_0(A, \ppp) = \bigg\{\begin{pmatrix}
        a & b \\ c & d
    \end{pmatrix} \in \SL_2(A) :  c \in \ppp \bigg\}, \]
where $A$ is a ring of $S$-integers, not totally imaginary, $|S| \geq 2$, and $\ppp$ is a nonzero prime ideal. We generalize the main part of \cite[Theorem C]{BBT2025-1}, which computes the abelianization of $\Gg_0(\z[1/n], \ppp)$ under some divisibility conditions. The first homology of the principal congruence subgroups was studied before by \cite{lee-szczarba}, but for $A = \z$ and matrices in $\SL_n(A)$, $n \geq 3$.

Our strategy consists of using the classic Congruence Subgroup Property \cite{serre1970} to reduce the problem to the computation of $\Gg(A/\ppp^{\alpha}, \ppp/\ppp^{\alpha})^{\ab}$, $\Gg_1(A/\ppp^{\alpha}, \ppp/\ppp^{\alpha})^{\ab}$ and $\Gg_0(A/\ppp^{\alpha}, \ppp/\ppp^{\alpha})^{\ab}$. For this reason, Section \ref{sec:local} is dedicated to describing the abelianization of congruence subgroups over a local ring $A$, with maximal ideal $\m_A$. The main result of this part relates $\Gg(A, \m_A)^{\ab}$ to the additive group of the Lie algebra $\mathfrak{sl}_2(\m_A/\m_A^2)$, formed by matrices with trace $0$ and entries in $\m_A/\m_A^2$. Notably, there is a decomposition
\[ \mathfrak{sl}_2(\m_A/\m_A^2) = \mathfrak{n}^{-} \oplus \mathfrak{h} \oplus \mathfrak{n}^{+},\]
where $\mathfrak{n}^-$, $\mathfrak{n}^{+}$ and $\mathfrak{h}$ are, respectively, the subalgebras formed by strictly lower triangular, strictly upper triangular, and diagonal matrices in $\mathfrak{sl}_2(\m_A/\m_A^2)$ \cite[p. 35]{humphreys}. Again, we see them as additive groups. 

%{color{red}{additive groups}} generated, respectively, by matrices of the form $\Ee_{12}(a), \Ee_{21}(b)$ and $D(1+c)$, for $a, b , c\in \m_A/\m_A^2$ {\color{red}{what are these matrices, specially $D(1+c)$? what is $\Ee_{21}(x)$ 
%for $x \in \m_A/\m_A^2$? ect.}}.

\begin{introthm}[\ref{H1(Gammaloc)}]
    Let $A$ be a local ring such that $\char(A/\m_A) \neq 2$. Then 
        \[
        \Gamma(A,\m_A)^\ab \simeq \mathfrak{sl}_2(\m_A/\m_A^2).
        \]
\end{introthm}

With this result in hand, we can use a five-term exact sequence coming from the Lyndon-Hochschild-Serre spectral sequence to compute $\Gg_1(A, \m_A)^{\ab}$ and $\Gg_0(A, \m_A)^{\ab}$ from $\Gamma(A,\m_A)^\ab$. The crucial advantage of the representation passing through the Lie algebra is that it makes it easier to understand the coinvariants of the action of $A/\m_A$ on $\Gg(A, \m_A)$.

In Section \ref{sec3}, we apply the local results to determine the first integral homology of the congruence subgroups with entries on a ring of $S$-integers. For the following theorems in this introduction, let $A = \OO_{K, S}$ be a Dedekind domain of arithmetic type with infinitely many units, not totally imaginary, and let $\ppp$ be a nonzero prime ideal of $A$. $\kappa(\ppp)$ denotes the residue field $A/\ppp$. Note that $H_1(\SL_2(A), \z)$ is known (\cite[Theorems 3.1 and 3.2]{BBT2025-2}). The abelianizations of the congruence subgroups $\Gg(A, \ppp), \Gg_1(A, \ppp)$ and $\Gg(A, \ppp)$ are explicitly given by the following theorems.

\begin{introthm}[\ref{H1(GammaGlob1)}]\leavevmode
    \begin{itemize}
        \item [(i)] If $|\kappa(\ppp)| \geq 4$ and $\char(\kappa(\ppp)) \neq 2$, then
        \[H_1(\Gg(A, \ppp), \z) \simeq \kappa(\ppp)^3\oplus H_1(\SL_2(A), \z).
        \]
        \item [(ii)] If $|\kappa(\ppp)| =3$, then
        \[
        H_1(\Gg(A, \ppp), \z) \simeq \kappa(\ppp)^2\oplus H_1(\SL_2(A), \z).  
        \]
    \end{itemize}
\end{introthm}

\begin{introthm}[\ref{Gg_1Glob1}]
    If $|\kappa(\ppp)| \geq 3$, then
     \[H_1(\Gg_1(A, \ppp), \z) \simeq \kappa(\ppp) \oplus H_1(\SL_2(A), \z).
        \]
\end{introthm}
\newpage
\begin{introthm}[\ref{Gg_0Glob}]\leavevmode
    \begin{itemize}
        \item [(i)] If $|\kappa(\ppp)| \geq 4$, then
        \[
        H_1(\Gg_0(A, \ppp), \z) \simeq \kappa(\ppp)^\times \oplus H_1(\SL_2(A), \mathbb{Z}).
        \]
        \item [(ii)] If $|\kappa(\ppp)|=3$, then
        \[
        H_1(\Gg_0(A, \ppp), \z)\simeq \kappa(\ppp) \oplus \kappa(\ppp)^\times \oplus H_1(\SL_2(A), \z).  
        \]
    \end{itemize}
\end{introthm}

Our findings also provide cohomological data: the torsion subgroup of the second integral cohomology is isomorphic to the torsion of the first homology. 
\begin{introprp}[\ref{torsion-prp}]
    Assume $|S| \geq 2$ if $K$ is a number field and $|S| \geq 3$ if $K$ is a function field. Let $\hat{\Gg}(A, \ppp)$ denote any of the three groups $\Gg(A, \ppp)$, $\Gg_1(A, \ppp)$ or $\Gg_0(A, \ppp)$. Then,
    \[H^{2}(\hat{\Gg}(A, \ppp), \z)_{\operatorname{tor}} \simeq H_1(\hat{\Gg}(A, \ppp), \z)_{\operatorname{tor}}. \]
\end{introprp}

%Let $\hat{\Gg}(A, \ppp)$ denote any of the groups $\Gg(A, \ppp)$, $\Gg_1(A, \ppp)$ or $\Gg_0(A, \ppp)$.
%\begin{introthm}[\ref{torsion-thm}]
%    Assume $|S| \geq 2$ if $K$ is a number field and $|S| \geq 3$ if $K$ is a function field. Then
%    \[H^{2}(\hat{\Gg}(A, \ppp), \z)_{\operatorname{tor}} \simeq H_1(\hat{\Gg}(A, \ppp), \z).\]
%\end{introthm}

\vspace{0.5cm}

\noindent\textbf{Acknowledgment.} We sincerely thank Behrooz Mirzaii for perusing the manuscript, and for his helpful comments and discussions about this article. The contributions of the first, second and fourth authors to this work, in the order they appear in the title, were made possible by CAPES (Coordena\c{c}\~ao de Aperfeiçoamento de Pessoal de N\'ivel Superior) fellowships 
(grant numbers 88887.136061/2025-00, 88887.995616/2024-00 and 88887.955905/2024-00). The third author is supported by a research grant (VIL54509) from VILLUM FONDEN.

\section{\texorpdfstring{$\SL_2$}{Lg} over Dedekind domains of arithmetic type} \label{sec1}

We start by recalling the definition of Dedekind domains of arithmetic type and known results about their structure. Then, we present properties relating subgroups generated by elementary matrices to congruence subgroups in $\SL_2$. We end the section by connecting this to the study of $\operatorname{GE}_2$-rings. For a more complete background on Dedekind domains of arithmetic type, see \cite[\S 1]{BBT2025-2}.

\subsection{Dedekind domains of arithmetic type}  

A global field $K$ is either a finite number field or a function field with one variable with coefficients on a finite field. If $\char (K) = 0$ (resp. $\char (K) > 0$), then the ring of algebraic integers of $K$, denoted by $\OO_K$, is the subring of all elements of $K$ that are integral over $\z$ (resp. over $\F_q[t]$). It is well established in the literature that $\OO_K$ is a Dedekind domain \cite[Theorem 6.24]{keune2023}. For $\ppp \in \mathrm{Spec}(\mathcal{O}_{K})$, we can define a discrete valuation $v_{\ppp}: K^{\times} \rightarrow \mathbb{Z}$ on $K$, called the \textit{$\ppp$-adic valuation} (see \cite[Definition 2.37]{keune2023}). These $\ppp$-adic valuations define norms on $K$, the so-called \textit{non-archimedean norms}, given by:
\[
|x|_{\ppp}:=\frac{1}{N(\ppp)^{v_{\ppp}(x)}}.
\]
On the other hand, every real embedding $K \hookrightarrow \mathbb{C}$ gives a norm on $K$, as does every pair of complex embeddings. These norms are called \textit{infinite places} or \textit{infinite primes}. We denote by $S_\infty$ the set of all infinite primes.

\begin{dfn}
Let $K$ be a global field and $S$ a finite nonempty set of primes containing $S_\infty$. We define:
$$
\mathcal{O}_{K,S}=\{a \in K^{\times} \text{ }| \text{ }v_{\ppp}(a)\geq 0 \text{ for all } 
\ppp \notin S \}. 
$$
A ring $A$ is called a \textit{Dedekind domain of arithmetic type}, or a \textit{ring of $S$-integers}, whenever $A=\OO_{K,S}$. 
If $S = S_{\infty}$ and every prime in $S$ comes from an imaginary embedding $K \hookrightarrow \mathbb{C}$, we say that $A$ is \textit{totally imaginary}.
\end{dfn}

By the classic Dirichlet Unit Theorem, the rank of the group of units in a Dedekind domain of arithmetic type is given in terms of $|S|$.

\begin{thm}[$S$-unit Theorem, Dirichlet, Hasse, Chevalley]\label{dir}
Let $A=\OO_{K,S}$ be a Dedekind domain of arithmetic type. Denote by $\aa$ the group of units of 
$A$ and let $\mu(A)$ be the group of roots of unity in $A$. Then
\[
\aa \simeq \mu(A) \oplus \z^{|S|-1},
\]
In particular, $\aa$ is always a finitely generated abelian group. Moreover, $A$ has infinitely 
many units if and only if $|S|\geq 2$.
\end{thm}
\begin{proof}
See \cite[Theorem 5.3.10]{weiss1976} for the general case, and \cite[Theorem 6.31]{keune2023} 
for the number field case.
\end{proof}

If $\ppp$ is a nonzero prime ideal of $A$, we denote the residue field modulo $\ppp$, $A/\ppp$, by $\kappa(\ppp)$. It is well known that, for each $\alpha \geq 1$, the ring $A/\ppp^{\alpha}$ is a finite local principal ideal ring, with maximal ideal $\ppp/\ppp^{\alpha}$ (see \cite[Exercise 5, Chap. III]{neukirch} and \cite[\S 1]{BBT2025-2}).

\subsection{Congruence subgroups and elementary matrices}\label{subsec:cong}
Let $A$ be a commutative ring. The special linear group of order $n$ over $A$, denoted by $\SL_n(A)$, is the multiplicative group of $n \times n$ matrices over $A$ with determinant equal to $1$. Given an ideal $I$ of $A$, the natural projection $\pi: A \rightarrow A/I$ induces a homomorphism 
\[ \pi^{(n)}_{*}: \SL_n(A) \rightarrow \SL_n(A/I).\]

The kernel of this map, denoted by $\Gamma^{(n)}(A, I)$, or $\SL_n(A, I)$, is called the \textit{principal congruence subgroup} of level $I$ in $\SL_n(A)$; it comprises all $n \times n$ matrices congruent to the identity modulo $I$, with determinant $1$. A \textit{congruence subgroup} of $\SL_n(A)$ is any subgroup that contains $\Gamma^{(n)}(A, I)$, for some nonzero ideal $I$ of $A$. It is clear that principal congruence subgroups are normal in $\SL_n(A)$ and that congruence subgroups are of finite index when $A = \mathcal{O}_{K, S}$, since $A/I$ is finite for $I \neq (0)$.

\textit{Elementary matrices} in $\SL_n(A)$ are denoted by $\Ee_{ij}(a)$ (with $i \neq j$) and differ from the identity only by the $(i,j)$-entry, where it has the element $a \in A$. The normal subgroup generated by all elementary matrices is denoted by $\Ee_n(A)$. We can also define elementary subgroups relative to an ideal: $\Ee_n(I)$ is the normal subgroup of $\SL_n(A)$ generated by the $\Ee_{ij}(a)$, with $a \in I$.

\begin{dfn}
    Let $A$ be a commutative ring and $I$ an ideal of $A$. We define the group $\SK_1(A, I)$ as:
    \[ \SK_1(A, I) := \frac{\SL(A, I)}{\Ee(I)},\]
    where $\SL(A, I):=\operatorname{colim}\Gamma^{(n)}(A, I)$ and $\Ee(I):=\operatorname{colim}{\Ee}_n(I)$. Here, $\operatorname{colim}$ can be thought as taking infinite union over $n$.
\end{dfn}

The following result describes $\SK_1(A, I)$ for the case where $A$ is a Dedekind domain of arithmetic type. For $x \in \mathbb{R}$, let $[x]$ be the greatest integer less than or equal to $x$. We define 
\[
[x]_{[0, n]} \coloneqq \inf(\sup(0, [x]), n),
\]
which is the integer nearest to $[x]$ within the interval $[0, n]$.

\begin{thm}[Bass-Milnor-Serre]\label{bass-milnor-serre}
    Let $A$ be a Dedekind domain of arithmetic type. 
    \begin{itemize}
        \item [(i)] If $A$ is not totally imaginary, then 
        \[\SK_1(A, I) = 1,\] for any ideal $I$ of $A$.

         \item [(ii)] If $A$ is totally imaginary and $m$ is the number of roots of unity in $A$, then
         \[\SK_1(A, I) = \mu_r,\]
         where $\mu_r$ is the group of $r$-th roots of unity in $A$, and $r$ is determined in terms of $I$ by the formula
         \[v_p(r) = \min_{\mathfrak{p} \mid p}\Bigg[  \frac{v_{\mathfrak{p}}(I)}{v_\mathfrak{p}(p)} - \frac{1}{p - 1} \Bigg]_{[0, v_p(m)]}.\]
        
    \end{itemize}
\end{thm}
\begin{proof}
    See \cite[Corollary 4.3]{bms1967}.
\end{proof}

Throughout this paper, whenever we omit the index of $\Gamma^{(n)}(A, I)$, it is understood that $n =2$, which is the case we will be mostly dealing with. %For matrices of order $2$, we can also define these groups relative to two nonzero ideals $I_1$ and $I_2$ of $A$, following \cite{vas1972}. Let

%\[
%\widetilde{\Gamma}(I_1, I_2) :=\bigg\{ {\mtxx a b c d} \in 
%\SL_2(A) : b \in I_1, c \in I_2, a-1, d-1 \in I_1I_2\bigg\}.
%\]
% Similarly, $\Ee_2(I_1, I_2)$ is the subgroup generated by matrices of the form 
%\[
%E_{12}(a):=\begin{pmatrix}
%1 & a\\
%0 & 1
%\end{pmatrix} \ \ \ \text{ and } \ \ \ 
%E_{21}(b):=\begin{pmatrix}
%1 & 0\\
%b & 1
%\end{pmatrix},
%\]
% with $a \in I_1$ and $b \in I_2$. It is clear that
%\[
%\Ee_2(I_1,I_2) \se \widetilde{\Gamma}(I_1, I_2).
%\]

%\begin{thm}[Vaserstein, Liehl]\label{VL}
%Let $A = \mathcal{O}_{K,S}$ be a Dedekind domain of arithmetic type with $|S| \geq 2$. Then $\Ee_2(I_1,I_2)$ is normal in $\widetilde{\Gamma}(I_1, I_2)$ and
%\[
%\frac{\widetilde{\Gamma}(I_1, I_2)}{\Ee_2(I_1,I_2)} = \SK_1(A, I_1I_2).
%\]
%Moreover, $\widetilde{\Gamma}(I_1, I_2)$ has finite index in $\SL_2(A)$.
%\end{thm}
%\begin{proof}
%See \cite[page 321, Theorem 1]{vas1972} and \cite[\S4]{liehl1981}. 
%\end{proof}

\begin{thm}[Serre]\label{gamma-over-EI}
    Let $A = \mathcal{O}_{K,S}$ be a Dedekind domain of arithmetic type with $|S| \geq 2$. Then
    \[ \frac{\Gamma(A, I)}{\Ee_2(I)} \simeq \SK_1(A, I).\]
\end{thm}
\begin{proof}
    See \cite[Theorem 2]{serre1970}.
\end{proof}

%\begin{cor}\label{prime-mu-r}
%    If $A$ is a Dedekind domain of arithmetic type with infinitely many units and $\ppp \in Spec(A)$, then
%    \[ \Gamma(A, \ppp) = \Ee_2(\ppp)\]
%\end{cor}
%\begin{proof}
%    This is a consequence of Corollary \ref{gamma-over-EI} and Theorem \ref{bass-milnor-serre}. We may assume that $A$ is totally imaginary, as otherwise the result follows immediately. Since $\mu_{r}$ is a cyclic group of order $r$, it suffices to show that $v_p(r) = 0$ for every integer prime $p$.

%    Let $p$ be a fixed prime number. If there exists some prime ideal $\qqq \neq \ppp$ diving $(p)$, we have 
%    \[ \Bigg[  \frac{v_{\qqq}(\ppp)}{v_\qqq(p)} - \frac{1}{p - 1} \Bigg]_{[0, v_p(m)]} = \Bigg[ - \frac{1}{p - 1} \Bigg]_{[0, v_p(m)]}  = 0.\]
%    Hence,
%    \[ 0 = \Bigg[  \frac{v_{\qqq}(\ppp)}{v_\qqq(p)} - \frac{1}{p - 1} \Bigg]_{[0, v_p(m)]} \geq v_p(r) \geq 0.\]

%    Now, assume $(p) = \ppp^{\alpha}$, for some $\alpha \geq 1$. In this case, 
%    \[ v_p(r) = \Bigg[  \frac{v_{\ppp}(\ppp)}{v_\ppp(p)} - \frac{1}{p - 1} \Bigg]_{[0, v_p(m)]} = \Bigg[  \frac{1}{\alpha} - \frac{1}{p - 1} \Bigg]_{[0, v_p(m)]}.\]
%    Since $1/\alpha  - 1/(p -1) <1$, we conclude that $v_p(r) = 0$.
%\end{proof}

The next theorem is due to Serre and plays a key role in this work, by establishing the \textit{Congruence Subgroup Property}.

\begin{thm}[Congruence Subgroup Property]
    If $A = \mathcal{O}_{K,S}$, $|S| \geq 2$, and $A$ is not totally imaginary, then every subgroup of finite index is a congruence subgroup. 
    %If $A = \mathcal{O}_{K,S}$ and $|S| \geq 2$, then every subgroup of finite index of $\SL_2(A)$ contains $\Ee_2(I)$, for some nonzero ideal $I$. If, in addition, $A$ is not totally imaginary, then $\Ee_2(I) = \Gamma(A, I)$.
\end{thm}
\begin{proof}
    From \cite[Proposition 1]{serre1970}, every subgroup of finite index in $\SL_2(A)$ contains $\Ee_2(I)$, for some nontrivial ideal $I$. By Theorem \ref{gamma-over-EI}, $\Ee_2(I) = \Gg(A, I)$.
\end{proof}

%Serre proved that, if $A = \mathcal{O}_{K,S}$ and $|S| \geq 2$, then every subgroup of finite index of $\SL_2(A)$ contains $\Ee_2(I)$, for some nonzero ideal $I$ (see \cite[Proposition 1]{serre1970}). If, in addition, $A$ is not totally imaginary, then $\Ee_2(I) = \Gamma(A, I)$, establishing the \textit{Congruence Subgroup Property}: every subgroup of finite index in $\SL_2(A)$ is a congruence subgroup. {\color{red}{State this part as a theorem!!}}

In addition to the elementary matrices, another class of matrices plays a crucial role in this work, namely:
\[
D(u):=\begin{pmatrix}
    u & 0\\
    0 & u^{-1}
\end{pmatrix},
\]
where $u \in A^{\times}$. The following lemma establishes key relations concerning the matrices $\Ee_{12}(a)$, $\Ee_{21}(b)$, and $D(u)$.

\begin{lem} \label{elem}
    Let $A$ be a commutative ring. Then, for all $a, b \in A$ and $u \in A^\times$, we have:

    \begin{itemize}
    
        \item [(i)] $[\Ee_{12}(-a), D(u)]=\Ee_{12}(a(u^2-1))$;

         \item [(ii)] $[\Ee_{21}(-b), D(u^{-1})]=\Ee_{21}(b(u^2-1))$.
        
    \end{itemize}
\end{lem}

\begin{proof}
    The proof follows from routine calculations.
\end{proof}

This work focuses on the study of $\Gg(A, I)$ and the two following congruence subgroups:
\[ \begin{aligned}
    \Gamma_1(A, I) &:=  \bigg\{\begin{pmatrix}
        a & b \\ c & d
    \end{pmatrix} \in \SL_2(A) : a-1,\ d-1,\ c \in I \bigg\},\\
    \Gamma_0(A, I) &:= \bigg\{\begin{pmatrix}
        a & b \\ c & d
    \end{pmatrix} \in \SL_2(A) : c \in I   \bigg\}.
\end{aligned}\]
Throughout this paper, we write $\hat{\Gg}(A, I)$ to denote any of the three groups: $\Gg(A, I)$, $\Gg_1(A, I)$, or $\Gg_0(A, I)$. Observe that they are related by the following group extensions:

\[ 1 \rightarrow \Gg(A, I) \rightarrow \SL_2(A) \overset{\pi_*}{\rightarrow} \SL_2(A/I) \rightarrow 1,\]
\[ 1 \rightarrow \Gg(A, I) \rightarrow \Gg_1(A, I) \overset{\mu}{\rightarrow} A/I \rightarrow 0,\]
\[ 1 \rightarrow \Gg_1(A, I)\rightarrow \Gg_0(A, I) \overset{\lambda}{\rightarrow} (A/I)^{\times} \rightarrow 1,\]
where $\pi_{*}$ is the natural map defined in Subsection \ref{subsec:cong} (for $n = 2$), $\mu$ is defined by $\begin{pmatrix}
    a & b \\ c & d 
\end{pmatrix} \mapsto \overline{b}$, and $\lambda$ is defined by $\begin{pmatrix}
    a & b\\ c & d
\end{pmatrix} \mapsto \overline{a}$.
%{\color{red}{(It is a good idea to the exact sequence involving $\Gamma, \Gamma_1, \Gamma_0, \SL_2$)}}

\subsection{$\operatorname{GE}_2$-rings}

\begin{dfn}
A ring $A$ is called a $\operatorname{GE}_2$-ring whenever $\SL_2(A) = \Ee_2(A)$. 
\end{dfn}

Being a $\operatorname{GE}_2$-ring endows $\SL_2(A)$ with several useful structural properties. We refer the reader to the classical works of Cohn \cite{cohn1966} and Silvester \cite{silv1982} for further details. To provide specific examples of $\operatorname{GE}_2$-rings that will be used later, we present two fundamental results as follows.

\begin{lem}
\label{GE1}
    Any semilocal ring is a $\operatorname{GE}_2$-ring.
\end{lem}

\begin{proof}
    See \cite[p. 245]{silv1982} for the general case, and \cite[Theorem 4.1]{cohn1966} for the local case.
\end{proof}

\begin{lem} \label{prod-ge2}
    The direct product of a finite number of $\operatorname{GE}_2$-rings is again a $\operatorname{GE}_2$-ring.
\end{lem}

\begin{proof}
\label{GE2}
    See \cite[Theorem 3.1]{cohn1966}. 
\end{proof}

From these two results, we obtain: 

\begin{cor}
    \label{GE2I}
    The quotient of a Dedekind domain by a nontrivial ideal is a $\operatorname{GE}_2$-ring.
\end{cor}

\begin{proof}
    Let $A$ be a Dedekind domain and $(0) \neq I \leq A$. We can write $I$ uniquely as
    \[
    I = \ppp_1^{\alpha_1} \ppp_2^{\alpha_2} \cdots \ppp_n^{\alpha_n},
    \]
   where $\ppp_i \in \Spec(A)$ are disjoint primes. Since each quotient $A/\ppp_i^{\alpha_i}$ is a local ring, Lemma~\ref{GE1} guarantees that these quotients are $\operatorname{GE}_2$-rings. By Lemma~\ref{prod-ge2}, it immediately follows that the same holds for $A/I$.
\end{proof}

%\begin{rem}
    %If $A$ is a Dedekind domain of arithmetic type, it follows that $A/I$ is finite \cite[Lemma 1.2]{BBT2025-2}, and hence a $\GE_2$-ring \cite[Corollary 2.3]{BBT2025-2}. {\color{red}{maybe remove this remark!!}}
%\end{rem}

A consequence of Serre's Theorem \ref{gamma-over-EI} is that Dedekind domains of arithmetic type with infinitely many units are $\GE_2$.

\begin{cor}
\label{A-is-GE2}
    If $A$ is a Dedekind domain of arithmetic type with infinitely many units, then $A$ is a $\operatorname{GE}_2$-ring. 
\end{cor}

\begin{proof}
   Setting $I= A$, we obtain $\Gamma(A, A) = \SL_2(A)$. Since $\SK_1(A, A) = 1$ (by Theorem \ref{bass-milnor-serre}), Theorem \ref{gamma-over-EI} gives $\SL_2(A)=\Ee_2(A)$.
\end{proof}

%\begin{proof}
%Pick $I_1=A=I_2$ so that $\widetilde{\Gamma}(I_1, I_2) = \SL_2(A)$, 
%$\Ee_2(I_1,I_2) = \Ee_2(A)$ and $\SK_1(A, AA) = 0$ (since $A$ is not totally imaginary).
%\end{proof}

%\begin{rem}
%Let $A$ be a Dedekind domain of arithmetic type with infinitely many units. It is well-known that $\SL_2(A):=\ker(\det_2)$, where $\det_2: \GL_2(A) \arr \aa$ is the determinant homomorphism. By the above corollary, we have that $\Ee_2(A)$ is normal in 
%$\GL_2(A)$. This is a notably fact because there are examples (due to A. Suslin, see 
%\cite{sus1976}) of Dedekind domains for which $\Ee_2(R)$ is not normal in $\GL_2(R)$.
%\end{rem}

\begin{lem}\label{surj-GE2}
Let  $f: A \arr B$ be a surjective homomorphism of  rings. If $B$ is a 
$\GE_2$-ring, then the natural map $f_\ast:\SL_2(A) \arr \SL_2(B)$ is surjective.
\end{lem}
\begin{proof}
Since $B$ is a $\GE_2$-ring, $\SL_2(B)$ is generated by the elementary matrices $\Ee_{12}(b_1)$
and $\Ee_{21}(b_2)$, with $b_1, b_2\in B$. Let $a_i \in A$, such that $b_i=f(a_i), i=1,2$. Hence $f_\ast(\Ee_{kl}(a_i))=\Ee_{kl}(b_i)$ for $k,l \in \{1,2\}$, with $k\neq l$.
\end{proof}

\section{The abelianization of congruence subgroups of \texorpdfstring{$\SL_2$}{Lg} over local rings} \label{sec:local}

As we will see in Section \ref{sec3}, the problem of computing the abelianization of $\hat{\Gg}(A, \ppp)$, when $A$ is a Dedekind domain of arithmetic type and $\ppp$ is a prime ideal, can be reduced to computing the abelianization of $\hat{\Gg}(A/\ppp, \ppp/\ppp^\alpha)$.
%, $\Gg_1(A/\ppp^\alpha, \ppp/\ppp^\alpha)$ and $\Gg_0(A/\ppp^\alpha, \ppp/\ppp^{\alpha})$, respectively.

Notably, if $A$ is a Dedekind domain, $A/\ppp^\alpha$ is a local ring with maximal ideal $\ppp/\ppp^{\alpha}$. Hence, in this section, we investigate the structure of congruence subgroups of $\SL_2$ over local rings in order to compute $H_1(\Gamma(A, \m_A), \z)$, $H_1(\Gg_1(A, \m_A), \z)$ and $H_1(\Gg_0(A, \m_A), \z)$. Throughout this section, $A$ is always a commutative ring with $1 \neq 0$, and whenever $A$ is local we denote its maximal ideal by $\m_A$.

\subsection{Structure of congruence subgroups} Here, we present some structural results about the commutator $[\Gg(A, \m_A), \Gg(A,\m_A)]$ and the abelianization of congruence subgroups of $\SL_2(A)$. 

%The following three definitions are needed for Lemma \ref{quot}, which describes the quotients $\Gg^{(n)}(A, I)/E_n(I)$ for local rings. 
%\begin{dfn}
%    Let $A$ be a ring with an ideal $I$. We say that $\alpha = (a_1, ..., a_n) \in A^{n}$ is $I$-\textit{unimodular} if there is a homomorphism $f: A^{n} \rightarrow A$ such that $f(\alpha) = 1$ and $\alpha \equiv (1, 0, ..., 0)\ (\operatorname{mod}\ I)$.
%\end{dfn}
%\begin{dfn}
%    We say that a ring $A$ satisfies condition $\rm{SR}_n(A, I)$ if, given $m \geq n$ and an $I$-unimodular element $\alpha = (a_1, ..., a_m) \in A^m$, there exist $a_i' = a_i + b_ia_m$, with $b_i \in I$ ($1 \leq i < m$) such that $(a_1', ..., a_{m-1}') \in A^{m-1}$ is $A$-unimodular.
%\end{dfn}
%\begin{dfn}
%    We say that $A$ satisfies condition $\rm{SR}'_n(A, I)$ if $\GL_n(A, I):= \ker(\GL_n(A) \rightarrow \GL_n(A/I))$ acts transitively on the $I$-unimodular elements in $A^n$.
%\end{dfn}

\begin{lem}\label{quot}
    Let $A$ be a local ring and let $I$ be an ideal of $A$. Then, for any $n \geq 2$,
    \[\frac{\Gg^{(n)}(A, I)}{\Ee_{n}(I)} \simeq \SK_1(A, I).\]
\end{lem}
\begin{proof}
    This is a consequence of the isomorphism $\displaystyle\frac{\GL_n(A, I)}{\Ee_n(I)} \simeq \Kk_1(A, I)$, which follows from \cite[Proposition 3.4, Chap. V]{bass1968} and \cite[Theorem 4.2, Chap. V]{bass1968}. For further details, we refer the reader to \cite[\S3 and \S4, Chap. V]{bass1968}.
\end{proof}

%\begin{proof}
%    From \cite[Proposition 3.4, Chap. V]{bass1968}, $A$ satisfies $\rm{SR}_2(A, I)$ and $\rm{SR}_2'(A, I)$. Then, \cite[Theorem 4.2, Chap. V]{bass1968} says that these two conditions yield
%    \[ \frac{\GL_n(A, I)}{\Ee_n(I)} \simeq \Kk_1(A, I),\]
%    for any $n \geq 2$. Here, $\Kk_1(A, I)$ is defined as $\GL(A, I)/\Ee(I)$, where $\GL(A, I) := \operatorname{colim}\GL_n(A, I)$. From this, it is clear that
%    \[\frac{\Gamma^{(n)}(A, I)}{\Ee_n(I)} \simeq \SK_1(A, I).\]
%\end{proof}

\begin{lem} \label{E_2}
    Let $A$ be a local ring. Then, for any $k \geq 0$,
    \[
    \Gamma(A, \m_A^k)=\Ee_2(\m_A^k). 
    \]
\end{lem}

\begin{proof}
    Let $k \geq 0$. Lemma \ref{quot} tells us that
    \[\frac{\Gg(A, \m_A^k)}{\Ee_2(\m_A^k)} \simeq \SK_1(A, \m_A^k).\]
    If $A$ is a commutative ring and $I$ is an ideal contained in the Jacobson radical of $A$, then $\SK_1(A, I) \simeq 1$ \cite[Lemma 2.4]{wei2013}, from which the result follows.
    %Given $x \in \m_A^k$, we have $1 + x \notin \m_A$. Since $A$ is local, $1 + x$ is a unit. in a commutative ring, the condition $x \in I\implies 1 + x \in A^{\times}$ yields $\SK_1(A, I) \simeq 1$.
\end{proof}

\begin{prp} \label{m2-trivial}
    Let $A$ be a local ring. If $\m_A^2$ is trivial, then $\Gamma(A, \m_A)$ is abelian.
\end{prp}
\begin{proof}
    By Lemma \ref{E_2}, $\Gamma(A, \m_A)=\Ee_2(\m_A)$, and we know that $\Ee_2(\m_A)$ is generated by all conjugates of $\Ee_{12}(x)$ and $\Ee_{21}(y)$, with $x, y \in \m_A$. Thus, it suffices to show that any two generators commute. Let
    \[\rm M =\begin{pmatrix}
        a & b \\ c & d
    \end{pmatrix}, \rm N = \begin{pmatrix}
        e & f \\ g & h
    \end{pmatrix} \in \SL_2(A).\]
    Since $\m_A^2$ is trivial and
    \[ \begin{aligned}
        \rm ME_{12}(x)M^{-1} = \begin{pmatrix}
            1 - acx & a^2x\\ -c^2 x & 1 + acx
        \end{pmatrix}, \ \  
        \rm NE_{12}(y)N^{-1} = \begin{pmatrix}
            1 - egy & e^2y\\
            -g^2y & 1 + egy
        \end{pmatrix}
    \end{aligned},\]
    it is straightforward to see that 
    \[\rm ME_{12}(x)M^{-1}NE_{12}(y)N^{-1} = NE_{12}(y)N^{-1}ME_{12}(x)M^{-1}.\]
    %as both sides are equal to
    %\[\begin{pmatrix}
    %    1 - acx -egy & a^2x + e^2y\\ -c^2x -g^2y & 1 + acx + egy
    %\end{pmatrix}.\]
    From similar routine computations, we see that 
    \[\rm ME_{21}(x)M^{-1}NE_{21}(y)N^{-1} = NE_{21}(y)N^{-1}ME_{21}(x)M^{-1}\]
    and
    \[\rm ME_{12}(x)M^{-1}NE_{21}(y)N^{-1} = NE_{21}(y)N^{-1}ME_{12}(x)M^{-1}.\]
\end{proof}

The following lemma will be fundamental for all results of this section. 

\begin{lem}
\label{gamcom}
Let $A$ be a local ring. If $\char(A/\m_A) \neq 2$, then
    \[
        \Gamma(A, \m_A^2) \subseteq [\Gamma(A, \m_A), \Gamma(A, \m_A)].
    \]
\end{lem}

\begin{proof}
    Again, by Lemma \ref{E_2}, $\Gamma(A, \m_A^2)=\Ee_2(\m_A^2)$, which is generated by all conjugates of $\Ee_{12}(x)$ and $\Ee_{21}(y)$, with $x, y \in \m_A^2$. Since $[\Gamma(A, \m_A), \Gamma(A, \m_A)]$ is normal in $\SL_2(A)$ and $\Ee_{21}(y)$ is a conjugate of $\Ee_{12}(-y)$, it suffices to show that
    \[ \Ee_{12}(x) \in [\Gamma(A, \m_A), \Gamma(A, \m_A)], \text{ for all } x \in \m_A^2.\]
    Any element of $\m_A^2$ is of the form $x = \sum_{i=1}^n a_ib_i$, with $a_i, b_i \in \m_A$.
    Since
    \[ \Ee_{12}(x) = \prod_{i=1}^n \Ee_{12}(a_ib_i),\]
    we just have to prove that each $\Ee_{12}(a_i, b_i)$ lies in the commutator. By Lemma \ref{elem},
    \[ [\Ee_{12}(b_i), D(1+a_i)] = \Ee_{12}(-b_i(a_i^2 +2a_i)) = \Ee_{12}(-a_ib_i(a_i + 2)).\]
    Using the fact that $a_i + 2 \in A^{\times}$,
    \[[\Ee_{12}(-b_i(a_i +2)^{-1}), D(1+a_i)] = \Ee_{12}(a_ib_i).\]
    % the case 2 \in \m_A is still in the backup file
\end{proof}

%\begin{rem}
%    $2 \in \m_A$ if, and only if, $\char(A/\m_A) = 2$. 
%\end{rem}

\begin{rem}
    The case of $\char(A/\m_A) = 2$ is more complex, even if $\m_A$ is principal. Let $\m_A = (a)$. If $2 = ca$, with $c \in \m_A$, then we can ensure that $\Gamma(A, \m_A^3)$ is contained in the commutator: for any $a^3b \in \m_A^3$,
    \[ E_{12}(a^3b) = [D(1 +a), E_{12}(ab(1+c)^{-1})]. \]
    If $2 = ua$, where $u \in A^{\times}$, we obtain $\Gg(A, \m_A^4) \subseteq [\Gg(A, \m_A), \Gg(A, \m_A)$ by checking that
    \[ E_{12}(16b) = [D(1 + 2), E_{12}(2b)],\]
    for every $16b \in \m_A^4$.

    However, this information is not necessarily helpful, since $\m_A^3$ or $\m_A^4$ might be trivial. For this reason, it would be necessary to split the discussion in cases, based on the nilpotence degree and on whether $2 \in \m_A \setminus \m_A^2$, making it hard to see the general picture. This is not a problem for $\char(A/\m_A) \neq 2$: Proposition \ref{m2-trivial} shows that $\Gg(A, \m_A)$ is already abelian when $\m_A^2 = (0)$.
\end{rem}

\begin{lem}
    \label{Gg_0abcond}
    Let $A$ be a local ring with maximal ideal $\m_A$. If $|A/\m_A| \geq 4$, then
    %\[\Gg_0(A, \m_A)^\ab \simeq \Gg_0(A/\m_A^2, \m_A/\m_A^2). \]
    \[\Gg(A, \m_A) \subseteq [\Gamma_1(A, \m_A), \Gamma_1(A, \m_A)] \subseteq [\Gamma_0(A, \m_A), \Gamma_0(A, \m_A)].\]
\end{lem}

\begin{proof}
    First note that any matrix in $\Gamma(A, \m_A)$ can be decomposed as
    \[ \begin{pmatrix}
        1 + a & b\\ c &  1 +d 
    \end{pmatrix} = E_{21}(c(1+a)^{-1})D(1+a)E_{12}(b(1+a)^{-1}),\]
    where $a, b, c, d \in \m_A$. Hence, it suffices to prove that elementary matrices and diagonal matrices are contained in the commutator of $\Gamma_1(A, \m_A)$.

    If $|A/\m_A| \geq 4$, then there is a unit $u\in A$ such that $u^2-1$ is also a unit. To see this, note that the equation $\overline{u}^2 -\overline{1} = \overline{0}$ has at most two solutions in the field $A/\m_A$. Now, given $a \in A$ and $ b\in \m_A$, we may write them as: 
    \begin{align*}
        &a=a'(u^2-1), \text{ and } \\
        &b=b'(u^2-1),
    \end{align*}
    with $a' \in A$, $b' \in \m_A$. By Lemma \ref{elem},
    \[ \begin{aligned}
        &E_{12}(a) = [D(u^2 -1), E_{12}(-a')],\\
        &E_{21}(b) = [D(u^2 -1), E_{21}(-b')].
    \end{aligned}\]

    Finally, one can check that
    \[ D(1 + a) = E_{12}(1)E_{21}(a)E_{12}(-1(1 + a)^{-1})E_{21}(-a(1+a)).\]
    %Therefore, by Lemma \ref{elem}, we have $E_{12}(a), E_{21}(b) \in [\Gg_1(A, \m_A), \Gg_1(A, \m_A)]$.
    %Now, by Theorem \ref{VL} we get that $\Gg(A, \m_A) \subseteq [\Gg_0(A, \m_A), \Gg_0(A, \m_A)]$, in particular, $\Gg(A, \m_A^2) \subseteq \Gg_0(A, \m_A)$, hence, the result follows from Corollary \ref{quotient}. 
\end{proof}

As illustrated by the following results, our technique passes through simplifying the abelianization of certain subgroups of $\SL_2$ by quotienting principal congruence subgroups. For instance, if $\hat{\Gg}(A, I)$ denotes any of the groups $\Gg(A, I)$, $\Gg_0(A, I)$, or $\Gg_1(A, I)$, we can characterize the quotients $\hat{\Gg}(A, I)/\Gg(A, J)$ via matrices with entries in $A/J$. 

%\begin{equation}
%\label{gamm_i}
%    \Gg_i(A, I) \in \{\Gg(A, I), \Gg_1(A, I), \Gg_0(A, I)\}.
%\end{equation}

\begin{lem}
    \label{gamma_i/gamma}
    Let $A$ be a commutative ring and let $I,J$ be ideals of $A$ with $J \subseteq I$. If $A/J$ is a $\GE_2$-ring, then 
    \[
    \hat{\Gg}(A, I)/\Gg(A, J) \simeq \hat{\Gg}(A/J, I/J).  
    \]
\end{lem}

\begin{proof}
     Since $A/J$ is a $\GE_2$-ring, Proposition \ref{surj-GE2} gives that $\pi_*: \SL_2(A) \rightarrow \SL_2(A/J)$ is surjective. Now, consider the following commutative diagram: 
     \[
    \begin{tikzcd}
\hat{\Gg}(A,I) \arrow[r,   "\pi_*'"] \arrow[d, hook,  "i_2"'] & \hat{\Gg}(A/ J, I/J) \arrow[d, hook, "i_1"] \\
\SL_2(A) \arrow[r, two heads, "\pi_*"'] & \SL_2(A/J),
\end{tikzcd}
\]
    where $\pi_*'=\pi_\ast|_{\hat{\Gg}(A, I)}$. If $X \in \hat{\Gg}(A/J, I/J) =\im(i_1)$, then $\pi_*^{-1}(X) \in \hat{\Gg}(A, I) = \im(i_2)$, hence $\pi_*'$ is surjective. Furthermore, $\Gg(A, J) \subseteq \hat{\Gg}(A, I) \cap \ker(\pi_\ast')$, which implies that: 
    \[
    \hat{\Gg}(A, I)/\Gg(A, J) \simeq \hat{\Gg}(A/J, I/J). 
    \]
\end{proof}

\begin{cor}
\label{quotient}
   Let $A$ be a local ring. Then, for all $k \geq 1$,
   \[  \hat{\Gg}(A, \m_A)/\Gg(A, \m_A^k)\simeq \hat{\Gg}(A/\m_A^k, \m_A/\m_A^k).\]
\end{cor}

%When we are studying the abelianization of the groups $\Gg_i(A/\m_A^k, \m_A/\m_A^k)$ there is some stability as we can see in next proposition: 

\begin{lem}
\label{h1ofcong}
    Let $A$ be a commutative ring and let $H$ be a subgroup of $\SL_2(A)$. If $J$ is an ideal of $A$ such that $\Gg(A, J) \subseteq [H,H]$, then
    \[H^{\ab} \simeq (H/\Gamma(A, J))^{\ab}. \] 
\end{lem}
\begin{proof}
%{\color{red}{I edited this proof!!}}
From $\Gg(A, J) \subseteq [H,H]$, we obtain the natural map
%    \[ \begin{aligned}
$f: H/\Gg(A, J) \rightarrow H^{\ab}$.
%        M \cdot \Gg(A, J) &\mapsto M\cdot [H, H]
 %   \end{aligned}\]
   This induces a well defined group homomorphism
    \[ \begin{aligned}
        f^{\ab}: (H/\Gg(A, J))^{\ab} &\rightarrow H^{\ab}.\\
%       (M \cdot \Gg(A, J)) \cdot [H, H]&\mapsto M \cdot [H, H].
    \end{aligned}\]
    Conversely, the quotient map $g: H \rightarrow H/\Gg(A,J)$ induces
    \[ \begin{aligned}
        g^{\ab}: H^{\ab} &\rightarrow (H/\Gg(A,J))^{\ab}.\\
%       M \cdot[H, H] &\mapsto (M\cdot \Gg(A, J)) \cdot [H, H].
    \end{aligned} \]
    It is clear that $g^{\ab}$ is the inverse of $f^{\ab}$, from which the result follows.
    %Immediately from the hypothesis we get that $(H/\Gamma(A, J))^\ab\twoheadrightarrow H^{\ab}$, while the map $H \rightarrow H/\Gg(A, J)$ induces the other surjection. \textcolor{blue}{I will correct this proof}
\end{proof}

A crucial consequence is expressed in the next proposition.

\begin{prp}
\label{Stab}
      Let $A$ be a local ring such that $\char(A/\m_A) \neq 2$. Then, for all $k \geq 2$,
      \[ \hat{\Gamma}(A/\m_A^k, \m_A/\m_A^k)^{\ab} \simeq \hat{\Gamma}(A/\m_A^2, \m_A/\m_A^2)^{\ab}.\]
\end{prp}

\begin{proof}
    Lemma \ref{gamcom} gives $\Gg(A, \m_A^k) \subseteq [\hat{\Gg}(A, \m_A), \hat{\Gg}(A, \m_A)]$, for all $k \geq 2$. Now, by Corollary \ref{quotient} and Lemma \ref{h1ofcong}, we have $\hat{\Gg}(A, \m_A)^\ab \simeq \hat{\Gg}(A/\m_A^k, \m_A/\m_A^k)^{\ab}$, for all $k \geq 2$.
    %Since $\Gg(A, \m_A^{k}) \subseteq \Gg(A, \m_A^2)$, for every $k \geq 2$, we also get the isomorphisms $\hat{\Gg}(A, \m_A)^\ab \simeq \hat{\Gg}(A/\m_A^k, \m_A/\m_A^k)^\ab$. 
\end{proof}

%This lemma provides an instrumental corollary, which will be extensively used in this paper. 

%\begin{cor} \textcolor{red}{!!!}
%\label{stab}
 %   Let $A$ be a local ring and let $H \leq \SL_2(A)$. If $[H,H]$ is a congruence subgroup, then there are natural numbers $n\geq k$ such that: 
 %   \[
 %   H^\ab \simeq (H/\Gg(A, \m_A^k))^\ab \simeq (H/\Gg(A, \m_A^{n}))^{ab}.
 %   \] 
%\end{cor}

%\begin{proof}
 %   By Lemma \ref{h1ofcong}, there is a natural number $k$ such that: 
  %  \[
  %  H^\ab \simeq (H/ \Gg(A, \m_A^k))^\ab,
  %  \]
   % where $\Gg(A, \m_A^k)\leq [H,H]$. Now, for all $n \geq k$, we have $\Gg(A, \m_A^n) \subseteq \Gg(A, \m_A^k) \subseteq [H,H]$. Thus, by the same arguments from Lemma \ref{h1ofcong}, we conclude that: 
   % \[
   % H^\ab\simeq (H/ \Gg(A, \m_A^n))^\ab.
   % \]    
%\end{proof}

\subsection{Abelianization of congruence subgroups}

With the results obtained in the previous section, we now can compute the abelianization of the congruence subgroups $\Gg_0$,  $\Gg_1$ and $\Gg$. We first start with $\Gg(A, \m_A)^\ab$ and then, using the five term exact sequence, we can compute the remaining ones. The strategy used for $\Gg(A, \m_A)^\ab$ is to relate it with the additive group of the Lie algebra $\Sl_2(\m_A/\m_A^2)$. %In fact, Theorem \ref{H1(Gammaloc)} states that these two groups are isomorphic, from which we can give a linear representation of $\Gg(A, \m_A)^{\ab}$ as: 
%\[
%\Gg(A, \m_A)^\ab \simeq \mathfrak{n}^-\oplus \mathfrak{h} \oplus \mathfrak{n}^+,
%\]
%where $\mathfrak{n}^{-}$ and $\mathfrak{n}^{+}$ denote the strictly lower and upper triangular matrices of $\mathfrak{sl}_2(\m_A/\m_A^2)$, while $\mathfrak{h}$ is formed by the diagonal matrices. In the next section, we characterize each summand of this group for the setting of Dedekind domains of arithmetic type using combinatorial arguments. 

\begin{thm}
\label{H1(Gammaloc)}
    Let $A$ be a local ring with $\char(A/\m_A) \neq 2$. Then, 
        \[
        \Gamma(A,\m_A)^\ab \simeq \mathfrak{sl}_2(\m_A/\m_A^2).
        \]
%{\color{red}{What these matrices mean? $\m_A/\m_A^2$ is a $A/\m_A$-vector space. So what $\Ee_{21}(x)$ means
%for $x \in \m_A/\m_A^2$? etc.}}
\end{thm}

\begin{proof}
    By Lemma \ref{gamcom}, $\Gg(A, \m_A^2)$ is contained in the commutator of $\Gg(A, \m_A)$. Using Lemma \ref{h1ofcong} and Corollary \ref{quotient}, we have: 
    \[
    \Gamma(A, \m_A)^{\ab}\simeq \Gamma(A/ \m_A^2, \m_A/ \m_A^2)^{\ab}. 
    \]
    
    From the fact that the determinant is $1$, one can check that any matrix in $\Gg(A/\m_A^2, \m_A/\m_A^2)$ can be written in the form
    \[ \begin{pmatrix}
            \overline{1}+\overline{a} & \overline{b}\\
            \overline{c} & \overline{1} - \overline{a}
        \end{pmatrix}.
    \]
    Thus, we define the following group homomorphisms: 
    \begin{align*}
        \Phi: \Gg(A/\m_A^2, \m_A/\m_A^2) &\rightarrow \mathfrak{sl}_2(\m_A/\m_A^2)\\
       \begin{pmatrix}
            \overline{1}+\overline{a} & \overline{b}\\
            \overline{c} & \overline{1} -\overline{a}
        \end{pmatrix}
        & \mapsto 
        \begin{pmatrix}
            \overline{a} & \overline{b}\\
            \overline{c} & \overline{-a}
        \end{pmatrix}
    \end{align*}
    and
     \begin{align*}
        \Psi: \mathfrak{sl}_2(\m_A/\m_A^2) &\rightarrow \Gg(A/\m_A^2, \m_A/\m_A^2)\\
       \begin{pmatrix}
           \overline{a} & \overline{b}\\
            \overline{c} & -\overline{a}
        \end{pmatrix}
        & \mapsto 
        \begin{pmatrix}
            1+\overline{a} & \overline{b}\\
            \overline{c} & 1 -\overline{a}
        \end{pmatrix},
    \end{align*}
    where $\mathfrak{sl}_2(\m_A/\m_A^2)$ denotes the additive group of the Lie Algebra $\mathfrak{sl}_2(\m_A/\m_A^2)$. Clearly, $\Phi$ and $\Psi$ are mutual inverses and hence 
\[
\Gamma(A/\m_A^2, \m_A/\m_A^2)\simeq \mathfrak{sl}_2(\m_A/\m_A^2).
\]
Since $\mathfrak{sl}_2(\m_A/ \m_A^2)$ is an abelian group, we get
\[
\Gg(A/\m_A^2, \m_A/\m_A^2)^{\ab} \simeq \mathfrak{sl}_2(\m_A/\m_A^2). 
\]
%By the definition of the map $\Phi$, we obtain: 
%\begin{align*}
%            &\Phi(\langle E_{21}(y) \mid y \in \m_A \rangle) = \mathfrak{n}^{-}\\
%            &\Phi(\langle D(a+1) \mid a \in \m_A\rangle) = \mathfrak{h}, \\
%            &\Phi(\langle  E_{12}(x) \mid x \in \m_A \rangle) = \mathfrak{n}^+.
%        \end{align*}
\end{proof}
\begin{rem}
    From the map $\Phi$, defined in the proof of Theorem \ref{H1(Gammaloc)}, we see that the subgroups $\langle E_{21}(y) \mid y \in \m_A/\m_A^2 \rangle$, $\langle D(a+1) \mid a \in \m_A/\m_A^2\rangle$, and $\langle  E_{12}(x) \mid x \in \m_A/\m_A^2 \rangle$ of $\SL_2(A/\m_A^2)$ correspond, respectively, to the terms $\mathfrak{n}^{-}$, $\mathfrak{h}$ and $\mathfrak{n}^{+}$, from the triangular decomposition of $\mathfrak{sl}_2(\m_A/\m_A^2)$.
\end{rem}

Once we have computed the group $\Gg(A, \m_A)^\ab$ we are able to obtain the group structure of $\Gg_1(A, \m_A)^\ab$ and $\Gg_0(A, \m_A)^\ab$ using the $5$-term exact sequence associated to the Lyndon-Hochschild-Serre spectral sequence \cite[Corollary 6.4, Chap. VII]{brown1994}, as we present in the next propositions. 
%For the next proposition recall $\mathfrak{n}^-\simeq \langle E_{21}(y) \mid y \in \m_A \rangle$ as in Theorem \ref{H1(Gammaloc)}. 

%{\color{blue}{Should we ask for a local ring with finite residue field? Since, in general, a local ring is not a Dedekind domain (it holds true iff the local ring is a DVR or a field). Also for Proposition \ref{H1GAM0loc}}.}

\begin{prp}
\label{H1GAM1loc}
    %Let $A$ be a  SPIR with $\char(A/\m_A)\neq 2$ and $A/\m_A \simeq \F_{p^f}$. Then, 
    Let $A$ be a local ring.
    \begin{enumerate}
        \item[(i)] If $|A/\m_A| \geq 4$, then
        \[
        \Gg_1(A, \m_A)^\ab \simeq A/\m_A.
        \]
        \item[(ii)] If $|A/\m_A| = 3$, then
        \[
        \Gg_1(A, \m_A)^\ab \simeq \mathfrak{n}^- \oplus A/\m_A. 
       \]
    \end{enumerate}     
\end{prp}

\begin{proof}\leavevmode
%{\color{red}{(Check this proof again!!!!!)}}
    \begin{enumerate}
        \item[(i)] By Lemma \ref{Gg_0abcond}, $\Gg(A, \m_A) \subseteq[\Gamma_1(A, \m_A), \Gamma_1(A, \m_A)]$. Thus, from Lemmas \ref{h1ofcong} and \ref{gamma_i/gamma}, we have:
        \[ \Gamma_1(A, \m_A)^{\ab} \simeq \Gg_1(A/\m_A, \m_A/\m_A)^{\ab} \simeq (A/\m_A)^{\ab} \simeq A/\m_A.\]

        %{\color{blue}{The residue field $A/\m_A$ is isomorphic to the ADDITIVE abelian group $\F_p^n$, where $p$ is a prime number. For finite fields, $|A/\m_A|$ is always a power of a prime number.}} {\color{red}{If $A/\m_A$ is finite, as additive group it is isomorphic to $\F_p^n$ and not $\z/p^n$!!}}

        \item[(ii)] Since $\char(A/\m_A)\neq 2$, Lemma \ref{gamcom} gives $\Gg(A, \m_A^2) \subseteq [\Gg_1(A, \m_A), \Gg_1(A,\m_A)]$. From Lemma \ref{h1ofcong} and Corollary \ref{quotient}, we get  $\Gg_1(A, \m_A)^\ab \simeq \Gg_1(A/\m_A^2, \m_A/\m_A^2)^\ab$.

        %{\color{red}{(Why? What happens to the rest of the exact sequence to the left? I see only three terms. The other term is $H_2(A/\m_A,\z)\simeq (A/\m_A) \wedge (A/\m_A)$ which is zero if $A/\m_A\simeq \F_3\simeq \z/3$.)}}

        %{\color{blue}{(In this case, the residue field $A/\m_A$ is a finite field with $3$ elements, i.e., $A/\m_A$ is $\F_3$, being isomorphic to $\z_3=\z/3\z$ (this holds in general: for a prime $p$, $\F_p\simeq\z_p=\z/p\z$), which is an additive finite cyclic group. Lastly, integral homologies of even index of finite cyclic groups all vanish \cite[pp. 58-59]{brown1994}.)}}
        
        Note that $A/\m_A \simeq \F_3$, which is an additive finite cyclic group, and hence $H_2(A/\m_A, \z)$ is trivial \cite[pp. 58-59]{brown1994}. Thus, by applying the $5$-term exact sequence \cite[Corollary 6.4, Chap. VII]{brown1994} to the group extension
        \[
        0 \rightarrow \Gg(A/\m_A^2, \m_A/ \m_A^2) \rightarrow \Gg_1(A/\m_A^2, \m_A/\m_A^2)
        \rightarrow A/\m_A \rightarrow 0,
        \]
        we get
        \[
        0 \rightarrow \Gg(A/\m_A^2, \m_A/\m_A^2)^\ab_{A/\m_A} \rightarrow \Gg_1(A/\m_A^2, \m_A/\m_A^2)^\ab \rightarrow A/\m_A \rightarrow 0,
        \]
        where the action of $A/\m_A$ on $\Gg(A/\m_A^2, \m_A/\m_A^2)^{\ab}$ is given by: 
        \[
        E_{12}(\xi)\overline{X} E_{12}(-\xi). 
        \]
        By Theorem \ref{H1(Gammaloc)} (i), we have:
        \[
        \Gg(A/\m_A^2, \m_A/\m_A^2)^\ab_{A/\m_A} \simeq \mathfrak{n}^-_{A/\m_A} \oplus \mathfrak{h}_{A/\m_A} \oplus \mathfrak{n}^+_{A/\m_A},
        \]
        with the action of $A/\m_A$ on each summand is transferred from $\mathfrak{sl}_2(\m_A/\m_A^2)$ by the map $\Phi$ of Theorem \ref{H1(Gammaloc)}. 
    
        Since $\mathfrak{n}^+\simeq \langle E_{12}(x) \mid x \in \m_A \rangle$ and $E_{12}(\xi) \overline{E_{12}(x)} E_{12}(-\xi) =\overline{E_{12}(x)}$ we get
        \[ 
       \mathfrak{n}^+_{A/\m_A} \simeq \mathfrak{n}^+.
        \]
        %\textcolor{blue}{please accept my apologizes, there was a type before. This equality leads to the fact that the Lie subalgebra $\mathfrak{n}^+$ has trivial coinvariant module under the action of $A/\m_A$, I have written something wrong}. 
        
        To evaluate the coinvariants of the action over $\mathfrak{h}$, consider the isomorphism $\mathfrak{h}\simeq \langle D(1+a) \mid a \in 
        \m_A \rangle$ and check that
        \[
        \Phi(E_{12}(\xi) \overline{D(1+a)} E_{12}(-\xi))-\Phi(D(1+a))=\begin{pmatrix}
            0 & -2\xi a\\
            0 & 0
            \end{pmatrix}.
        \]
        Similarly, for $\mathfrak{n}^-\simeq \langle E_{21}(y) \mid y \in \m_A\rangle$, we have
        \[
        \Phi(E_{12}(\xi)E_{21}(y)E_{12}(-\xi)) - \Phi(E_{21}(y))
        =
        \begin{pmatrix}
          \xi y & -\xi^2 y\\
          0  &  -\xi y
        \end{pmatrix},
        \]
        Since $\char(A/\m_A) \neq 2$, straightforward computations show that 
        \[ \langle E_{12}(\xi)ME_{12}(-\xi) : M \in \mathfrak{n}^{-} \oplus \mathfrak{h} \rangle \simeq \mathfrak{h} \oplus \mathfrak{n}^{+}\]
        Therefore,
        \[ \Gamma(A/\m_A^2, \m_A/\m_A^2)_{(A/\m_A)}  \simeq \frac{\mathfrak{n}^{-} \oplus \mathfrak{h} \oplus \mathfrak{n}^{+}}{\mathfrak{h} \oplus \mathfrak{n}^{+}} \simeq \mathfrak{n}^{-},\]
        and we conclude that 
        \[
        \Gg_1(A/\m_A)^\ab\simeq \mathfrak{n}^-\oplus A/\m_A.
        \]
    \end{enumerate}

\end{proof}

%For the group $\Gg_0(A, \m_A)$ we can replace the condition $\char(A/\m_A)\neq 2$ for $|A/\m_A| \neq 2$ as showed in the next lemma. 

\begin{prp}
\label{H1GAM0loc}
    Let $A$ be a local ring with finite residue field.
    \begin{itemize}
        \item [(i)] If $|A/\m_A| \geq 4$, then
        \[
        H_1(\Gg_0(A, \m_A), \z) \simeq (A/\m_A)^\times.
        \]
        \item [(ii)] If $|A/\m_A|=3$, then
        \[
        H_1(\Gg_0(A, \m_A), \z) \simeq \mathfrak{n}^{-} \oplus A/\m_A \oplus (A/\m_A)^{\times}. 
        % i changed \F_2 to \z/2 because i think it is more consistent, since we used \z/3 instead of \F_3
        \]
    \end{itemize}
\end{prp}

\begin{proof}
     Lemma \ref{gamcom} and Corollary \ref{quotient} give $\Gg_0(A, \m_A)^\ab \simeq \Gg_0(A/\m_A^2, \m_A/\m_A^2)^\ab$. Note that $(A/\m_A)^{\times}$ is a finite cyclic group \cite[Theorem 2.1.3]{bb2002}, and so $H_2((A/\m_A)^{\times}, \z)$ is trivial \cite[pp. 58-59]{brown1994}. Thus, by applying the $5$-term exact sequence \cite[Corollary 6.4, Chap. VII]{brown1994} to
     \[
     0 \rightarrow \Gg_1(A/\m_A^2, \m_A/\m_A^2) \rightarrow \Gg_0(A/\m_A^2, \m_A/\m_A^2) \rightarrow (A/\m_A)^\times \rightarrow 0,
    \]
    we obtain
    
     %applying five term exact sequence yields to the following: {\color{red}{(Why? What happens to the rest of the exact sequence to the left?
     %The other term is $H_2((A/\m_A)^\times,\z)$ which is zero if $A/\m_A$ is finite.)}}
     %{\color{blue}{Again, for finite fields, $(A/\m_A)^\times$ is a finite cyclic group}, since the multiplicative group of a finite field is a cyclic group \cite[Theorem 2.1.3]{bb2002}, and so the homology in question vanishes.}
     
     \[
     0 \rightarrow \Gg_1(A/\m_A^2, \m_A/\m_A^2)^{\ab}_{(A/\m_A)^{\times}} \rightarrow \Gg_0(A/\m_A^2, \m_A/\m_A^2)^{\ab} \rightarrow (A/\m_A)^\times \rightarrow 0.
      \]
      Now, observe that the action of $(A/\m_A)^\times$ over $\Gg_1(A/\m_A^2, \m_A/\m_A^2)^{\ab}$ is completely determined by the action of $\overline{D(u)}$ on $\overline{E_{12}(x)}$ and $\overline{E_{21}(y)}$ for $u \in (A/\m_A)^\times,\ x \in A/\m_A^2$ and $y \in \m_A/\m_A^2$.
      
      For $|A/\m_A| \geq 4$, there exists a unit $u \in A/\m_A$ such that $u^2 - 1$ is also invertible. Let $\tilde{u} \in A/\m_A^2$ such that $\overline{\tilde{u}} = u$ in the quotient $(A/\m_A^2)/(\m_A/\m_A^2)$. Given any $x \in A/\m_A^2$ we can write 
      \[ \overline{E_{12}(x)}  = \overline{E_{12}(x'(\tilde{u}^2-1))} = \overline{D(\tilde{u})}\cdot \overline{E_{12}(x')}\cdot \overline{D(\tilde{u}^{-1})} - \overline{E_{12}(x')},\]
      for some $x' \in A$. Similarly, given $y \in \m_A/\m_A^2$, there exists $y' \in \m_A/\m_A^2$ such that $y = y'(\tilde{u}^2 -1)$, and we obtain
      \[\overline{E_{21}(y)}  = \overline{E_{21}(y'(\tilde{u}^2-1))} = \overline{D(\tilde{u})}\cdot \overline{E_{21}(y')}\cdot \overline{D(\tilde{u}^{-1})} - \overline{E_{21}(y')}.\]
    Therefore,
    \[\Gg_1(A/\m_A^2, \m_A/\m_A^2)^{\ab}_{(A/\m_A)^{\times}} \simeq 0.\]
     % Therefore, Lemma \ref{elem} implies that $\Gg_1(A/\m_A^2, \m_A/\m_A^2)^\ab_{(A/\m_A)^\times}=0$ for $|A/\m_A| \geq 4$.
     
     If $|A/\m_A|=3$, $u^2 -1 = 0$ for any $u \in (A/\m_A)^{\times}$, so the action is trivial by Lemma \ref{elem} and we get
      \[
      \Gg_1(A/\m_A^2, \m_A/\m_A^2)^\ab_{(A/\m_A)^\times}\simeq \Gg_1(A/\m_A^2, \m_A/\m_A^2)^\ab \simeq \mathfrak{n}^{-} \oplus A/\m_A,
      \]
      where the last isomorphism is given by Theorem \ref{H1(Gammaloc)}. %Hence, in this case: 
      %\[
      %\begin{aligned}
      %    \Gg_0(A, \m_A)^\ab &\simeq \z/3 \oplus \z/3 \oplus (\z/3)^{\times} \\
      %    &\simeq \z/3 \oplus \z/3 \oplus \z/2
      %\end{aligned}
      %\]
\end{proof}

\section{The abelianization of congruence subgroups of \texorpdfstring{$\SL_2$}{Lg} over $S$-integers} \label{sec3}

In the preceding section, we computed the abelianization of congruence subgroups of $\SL_2(A)$ when $A$ is a local ring. We now apply these local results to characterize the abelianization of congruence subgroups over rings of $S$-integers. Finally, we use the first homology of these groups to determine the torsion part of the second cohomology.
%Specifically, we aim to describe the first integral homology of the congruence subgroups $\Gg(A, \ppp)$, $\Gg_1(A, \ppp)$ and $\Gg_0(A, \ppp)$, where $A = \mathcal{O}_{K, S}$ (not totally imaginary, with infinitely many units) and $\ppp \in \Spec(A)$ \textcolor{blue}{ is this part unnecessary, considering the introduction?}

\subsection{First homology}
%Again, let $\hat{\Gg}(A, I)$ denote one of the groups $\Gg(A, I)$, $\Gg_0(A, I)$ and $\Gg_1(A, I)$. 
We observe that if $\Gg(A, J) \subseteq  [\hat{\Gg}(A, I), \hat{\Gg}(A, I)]$, then $\Gg(A, J) \subseteq  \hat{\Gg}(A, I)$. In particular, we must have $J \subseteq I$. In Dedekind domains, this condition is equivalent to $I \mid J$. If $I=\ppp \in \Spec(A)$,  we can decompose $J$ as $J=\ppp^\alpha J'$, for some $J'\unlhd A$ coprime to $\ppp$. %For the next result, suppose that: 

\begin{thm}
\label{DDab}
    Let $A$ be a Dedekind domain and let $\ppp$ be a nonzero prime ideal such that $\Gg(A, J) \subseteq [\hat{\Gg}(A, \ppp), \hat{\Gg}(A, \ppp)]$, for some non-trivial ideal $J \unlhd A$. Then:
    \[
    H_1(\hat{\Gg}(A, \ppp), \z) \simeq H_1(\hat{\Gg}(A/\ppp^\alpha, \ppp/\ppp^\alpha), \z) \oplus H_1(\SL_2(A/J'), \z), 
    \]
    where $J = \ppp^{\alpha}J'$, $\alpha \geq 1$ and $J'$ is coprime to $\ppp$.
    %for $J'=J/\ppp^\alpha$. 
\end{thm}

\begin{proof}
    Since $\Gg(A, J) \subseteq [\hat{\Gg}(A, \ppp), \hat{\Gg}(A, \ppp)]$, by Lemmas \ref{h1ofcong} and \ref{gamma_i/gamma} we have the isomorphism:
    \[
     \hat{\Gg}(A, \ppp)^\ab \simeq (\hat{\Gg}(A, \ppp)/\Gg(A, J))^{\ab}. 
    \]
    Now we write $J=\ppp^\alpha J'$ with $\ppp^\alpha$ and $J'$ comaximals. There is a natural surjection
    \[
    \Omega: \SL_2(A) \twoheadrightarrow \SL_2(A/\ppp^\alpha) \times \SL_2(A/ J'). 
    \]
    Let $\Omega'=\Omega|_{\Gg_i(A, \ppp)}$ and consider the following commutative diagram: 
         \[
    \begin{tikzcd}
\hat{\Gg}(A, \ppp) \arrow[r,   "\Omega'"] \arrow[d, hook,  "i_2"'] & \hat{\Gg}(A/ \ppp^\alpha, \ppp/\ppp^\alpha) \times \SL_2(A/J') \arrow[d, hook, "i_1"] \\
\SL_2(A) \arrow[r, two heads, "\Omega"'] & \SL_2(\ppp/\ppp^\alpha) \times \SL_2(A/J').
\end{tikzcd}
\]

If we take an $X \in \hat{\Gg}(A/\ppp^\alpha, \ppp/\ppp^\alpha) \times \SL_2(A/J')$, then $\Omega^{-1}(i_1(X)) \in \operatorname{im}(i_2)$, implying that $\Omega'$ is surjective. Once $\Gg(A, J) \subseteq \hat{\Gg}(A, \ppp) \cap \ker(\Omega)$,  we obtain the following exact sequence: 
\[
0 \rightarrow \Gg(A, J) \rightarrow \hat{\Gg}(A, \ppp) \xrightarrow{\Omega} \hat{\Gg}(A/\ppp^\alpha, \ppp/\ppp^\alpha) \times \SL_2(A/J') \rightarrow 0.
\]
Therefore,
\[
\hat{\Gg}(A, \ppp)^\ab \simeq (\hat{\Gg}(A, \ppp)/ \Gg(A, J))^{\ab}\simeq \hat{\Gg}(A/\ppp^\alpha, \ppp/\ppp^{\alpha}) \oplus \SL_2(A/J')^\ab.
\]
In other words,
\[
    H_1(\hat{\Gg}(A, \ppp), \z) \simeq H_1(\hat{\Gg}(A/\ppp^\alpha, \ppp/\ppp^\alpha), \z) \oplus H_1(\SL_2(A/J'), \z).
\]   
\end{proof}

\begin{lem} \label{H1-SL2-A-J'}
    Let $A$ be a Dedekind domain of arithmetic type with infinitely many units, and suppose $J$ is an ideal such that $6 \mid J$ and $\Gg(A, J) \subseteq [\SL_2(A), \SL_2(A)]$. If $(0) \neq \ppp \in \Spec(A)$, let $J'$ be the ideal obtained by removing the power of $\ppp$ from $J$ (i.e. $J' = J\cdot \ppp^{-v_{\ppp}(J)}$).
    \begin{enumerate}
        \item[(i)] If $|\kappa(\ppp)| \geq 4$, then
        \[H_1(\SL_2(A), \z) \simeq H_1(\SL_2(A/J'), \z).\]
        \item[(ii)] If $|\kappa(\ppp)| = 3$, then
        \[ H_1(\SL_2(A), \z)\simeq H_1(\SL_2(A/J'), \z) \oplus \kappa(\ppp).\]
        \item[(iii)] If $|\kappa(\ppp)| = 2$, then
        \[ H_1(\SL_2(A), \z)\simeq H_1(\SL_2(A/J'), \z) \oplus A/\ppp^2.\]
    \end{enumerate}
\end{lem}
\begin{proof}
    Since $\Gg(A, J) \subseteq [\SL_2(A), \SL_2(A)]$, we can proceed as in Theorem \ref{DDab} and conclude that 
    \[\SL_2(A)^{\ab} \simeq \SL_2(A/J)^{\ab}.\]
    Factorizing $J = \qqq_1^{\beta_1}\cdots \qqq_n^{\beta_n}$ and using Chinese Reminder Theorem and the fact thats 
    $\SL_2$ and abelianization commute with finite product, we obtain
    \[ \SL_2(A/J)^{\ab} \simeq \prod_{i =1}^n\SL_2(A/\qqq_i^{\beta_i})^\ab.\]
    Each of the factors $A/\qqq_i^{\beta_i}$ is a local ring, so we can use \cite[Proposition 4.1]{B-E2024} to see that 
    \[ \SL_2(A/\qqq_i^{\beta_i})^\ab \simeq\begin{cases}
        A/\qqq_i^2, &\text{ if }|\kappa(\qqq_i)| = 2\\ 
        \kappa(\qqq_i), &\text{ if }|\kappa(\qqq_i)| = 3\\ 
        0, &\text{ if }|\kappa(\qqq_i)| \geq  4.\\ 
    \end{cases}\]
    Hence, for $J = \ppp^{\alpha}J'$, we have
    \[ \SL_2(A/J)^{\ab} \simeq \begin{cases}
        \SL_2(A/J')^{\ab}, &\text{ if } |\kappa(\ppp)| \geq 4;\\
        \SL_2(A/J')^{\ab} \oplus \kappa(\ppp), &\text{ if } |\kappa(\ppp)| =3;\\
        \SL_2(A/J')^{\ab} \oplus A/\ppp², &\text{ if } |\kappa(\ppp)| =2.
    \end{cases}
    \]

    %\cite[p.9-13, proof of Theorems 3.1 and 3.2]{BBT2025-2}
\end{proof}

Theorem \ref{DDab} and Lemma \ref{H1-SL2-A-J'} reduce the computation of $H_1(\hat{\Gg}(A, \ppp), \z)$ to evaluating the group $\allowbreak H_1(\hat{\Gg}(A/\ppp^\alpha, \ppp/\ppp^\alpha), \z)$, since the summand $H_1(\SL_2(A), \z)$ is already known from \cite[Theorems 3.1 and 3.2]{BBT2025-2}. To determine $\hat{\Gg}(A/\ppp^{\alpha}, \ppp/\ppp^\alpha)^\ab$, we return to the results of previous sections (Theorem \ref{H1(Gammaloc)}): it suffices to analyze $\mathfrak{n}^-, \mathfrak{h}$, and $\mathfrak{n}^+$. These groups admit a simple description, which we give in the following proposition.

\begin{prp}\label{m-over-m2}
    Let $A$ be a local ring with principal maximal ideal $\m_A$. Then,
    \[ \mathfrak{n}^{-} \simeq \mathfrak{h}\simeq \mathfrak{n}^{+} \simeq \m_A/\m_A^2.\]
    Moreover,
    \begin{enumerate}
        \item[(i)] If $A$ is not a field, then 
        \[ \m_A/\m_A^2 \simeq A/\m_A.\]
    \item[(ii)] If $A$ is a field, then
    \[ \m_A/\m_A^2 = 0.\]
    \end{enumerate}
\end{prp}
\begin{proof}
    %From Theorem \ref{H1(Gammaloc)}, $\Gamma(A, \m_A)^{\ab}$ can be decomposed as $\mathfrak{n}^{-} \oplus \mathfrak{h} \oplus \mathfrak{n}^{+}$.
    Each of the three terms is isomorphic to the abelian group $\m_A/\m_A^2$, through the homomorphisms given by
    \[ \begin{aligned}
        E_{12}(x) &\mapsto x,\\
        E_{21}(y) &\mapsto y,\\
        D(a + 1) &\mapsto a.
    \end{aligned}\]

    Note that by Nakayama's Lemma $\m_A = \m_A^2$ if and only if $A$ is a field. If $A$ is not a field and $\m_A = (a)$, one readily checks that the group homomorphism
    \[ \begin{aligned}
        \eta: A &\rightarrow \m_A/\m_A^2\\
        x &\mapsto \overline{ax}.
    \end{aligned}\]
    is surjective and $\ker(\eta) = \m_A$.
\end{proof}

\begin{thm}
    \label{H1(GammaGlob1)}
    Let $A$ be a Dedekind domain of arithmetic type with infinitely many units, not totally imaginary. If $(0) \neq \ppp \in \Spec(A)$ and $\char(\kappa(\ppp)) \neq 2$, then
    \begin{itemize}
        \item [(i)] For $|\kappa(\ppp)| \geq 4$, 
        \[H_1(\Gg(A, \ppp), \z) \simeq \kappa(\ppp)^3\oplus H_1(\SL_2(A), \z);
        \]
        \item [(ii)] For $|\kappa(\ppp)| =3$,
        \[
        H_1(\Gg(A, \ppp), \z) \simeq \kappa(\ppp)^2\oplus H_1(\SL_2(A), \z).  
        \]
    \end{itemize}
\end{thm}

\begin{proof}

If $A^{\times}$ is infinite and $A$ is not totally imaginary, $[\Gamma(A, \ppp), \Gg(A, \ppp)]$ is a congruence subgroup: it is a normal noncentral subgroup of $\SL_2(A)$, so we can use \cite[Theorem 3]{serre1970} and the Congruence Subgroup Property. Hence, there exists an ideal $I \subseteq \ppp$, such that $\Gg(A, I) \subseteq [\Gg(A, \ppp), \Gg(A, \ppp)]$. Letting $J=6\ppp^2I$, we still have $\Gg(A, J) \subseteq [\Gg(A, \ppp), \Gg(A, \ppp)]$. By Theorem \ref{DDab}, we obtain the isomorphism: 
\[
H_1(\Gg(A, \ppp), \z)\simeq H_1(\Gg(A/\ppp^\alpha, \ppp/\ppp^\alpha), \z) \oplus H_1(\SL_2(A,J'), \z),
\]
where $J = \ppp^{\alpha}J'$ and $J'$ is coprime with $\ppp$. 

Now, the result follows from Theorem \ref{H1(Gammaloc)}, Lemma \ref{H1-SL2-A-J'} and Proposition \ref{m-over-m2}.
% \cite[Theorem 3.1, Theorem 3.2]{BBT2025-2},
%Proposition \ref{Stab}. 
\end{proof}
%We have the global version for $H_1(\Gg_1(A, \ppp), \z)$:

\begin{thm}
\label{Gg_1Glob1}
     Let $A$ be a Dedekind domain of arithmetic type with infinitely many units, not totally imaginary. If $(0) \neq \ppp \in \Spec(A)$ and $|\kappa(\ppp)| \geq 3$, then
     \[H_1(\Gg_1(A, \ppp), \z) \simeq \kappa(\ppp) \oplus H_1(\SL_2(A), \z).
        \]
\end{thm}

\begin{proof}
    As in the proof of Theorem \ref{H1(GammaGlob1)}, we can employ Theorem \ref{DDab}, Propositions \ref{H1GAM1loc} and \ref{m-over-m2}, and Lemma \ref{H1-SL2-A-J'}.
    %As in the the proof of Theorem \ref{H1(GammaGlob1)} we can employ \cite[Theorem 3.1, Theorem 3.2]{BBT2025-2} and Propostion \ref{H1GAM1loc} to obtain the desired result. 
    
\end{proof}
%For $\Gg_0(A, \ppp)$ we have the following result provide $|A/\ppp| \neq 2$:

\begin{thm}
    \label{Gg_0Glob}
    Let $A$ be a Dedekind domain of arithmetic type with infinitely many units, not totally imaginary. If $(0) \neq \ppp \in \Spec(A)$, then 
    \begin{itemize}
        \item [(i)] For $|\kappa(\ppp)| \geq 4$, 
        \[
        H_1(\Gg_0(A, \ppp), \z) \simeq \kappa(\ppp)^\times \oplus H_1(\SL_2(A), \mathbb{Z}).
        \]
        \item [(ii)] For $|\kappa(\ppp)|=3$, 
        \[
        H_1(\Gg_0(A, \ppp), \z)\simeq \kappa(\ppp) \oplus \kappa(\ppp)^\times \oplus H_1(\SL_2(A), \z).  
        \]
    \end{itemize}
\end{thm}
\begin{proof}
    Once again, use Theorem \ref{DDab}, Propositions \ref{H1GAM0loc} and $\ref{m-over-m2}$, as well as Lemma \ref{H1-SL2-A-J'}.
\end{proof}

\begin{cor}
    Let $n > 1$ be a fixed integer and
    \[ A_n = \z[1/n] := \{ a/n^{k} : a \in \z, k \in \mathbb{\z}\}.\]
    Let $p$ be a prime, with $p \nmid n$, and set $\ppp = pA_n$. The first integral homology of the groups $\Gg(A_n, \ppp)$, $\Gg_1(A_n, \ppp)$ and $\Gg_0(A_n, \ppp)$ is described in Tables \ref{tab-p>3} and \ref{tab-p=3}:
    \begin{table}[h]
        \centering
        \renewcommand{\arraystretch}{1.5}
        \caption{First homology of $\Gg, \Gg_1$ and $\Gg_0$, with entries in $\z[1/n]$ and $p > 3$.}
        \begin{tabular}{c|c|c|c}
            Case & $H_1(\Gg(A_n, \ppp), \z)$ & $H_1(\Gg_1(A_n, \ppp), \z)$ &  $H_1(\Gg_0(A_n, \ppp), \z)$ \\\hline
            $2 \mid n$, $3 \mid n$ & $(\z/p)^3$ & $\z/p$ & $(\z/p)^{\times}$\\
            $2 \mid n$, $3 \nmid n$ & $(\z/p)^3 \oplus \z/3$ &  $\z/p \oplus \z/3$ & $(\z/p)^{\times} \oplus \z/3$ \\
            $2 \nmid n$, $3 \mid n$ & $(\z/p)^3 \oplus \z/4$ &  $\z/p \oplus \z/4$ & $(\z/p)^{\times} \oplus \z/4$ \\
            $2 \nmid n$, $3 \nmid n$ & $(\z/p)^3 \oplus \z/12$ & $\z/p \oplus \z/12$ & $(\z/p)^{\times} \oplus \z/12$\\
        \end{tabular}
        
        \label{tab-p>3}
    \end{table}

    \begin{table}[h]
        \centering
        \renewcommand{\arraystretch}{1.5}
        \caption{First homology of $\Gg, \Gg_1$ and $\Gg_0$, with entries in $\z[1/n]$ and $p = 3$.}
        \begin{tabular}{c|c|c|c}
            Case & $H_1(\Gg(A_n, \ppp), \z)$ & $H_1(\Gg_1(A_n, \ppp), \z)$ &  $H_1(\Gg_0(A_n, \ppp), \z)$ \\\hline
            $2 \mid n$, $3 \mid n$ & $(\z/3)^2$ & $\z/3$ & $\z/6$\\
            $2 \mid n$, $3 \nmid n$ & $(\z/3)^3$ &  $(\z/3)^2$ & $\z/6 \oplus \z/3$ \\
            $2 \nmid n$, $3 \mid n$ & $(\z/3)^2 \oplus \z/4$ &  $\z/12$ & $\z/6 \oplus \z/4$ \\
            $2 \nmid n$, $3 \nmid n$ & $(\z/3)^2 \oplus \z/12$ & $\z/3 \oplus \z/12$ & $\z/6 \oplus \z/12$\\
        \end{tabular}
        
        \label{tab-p=3}
    \end{table}
\end{cor}
\begin{proof}
    Use the divisibility of $n$ by $2$ and $3$ to determine $H_1(\SL_2(A_n), \z)$, 
    as described in \cite[Proposition 4.4]{B-E2024} and \cite[Example 3.4]{BBT2025-2}. Then apply Theorems \ref{H1(GammaGlob1)}, \ref{Gg_1Glob1} and \ref{Gg_0Glob}.
\end{proof}

\begin{rem}
    The hypothesis of infinitely many units and $A$ not being totally imaginary in Theorems \ref{H1(GammaGlob1)}, \ref{Gg_1Glob1}, and \ref{Gg_0Glob} can be replaced by the weaker condition that the respective commutators are congruence subgroups. 
\end{rem}
\begin{rem}
    Theorem \ref{Gg_0Glob} generalizes items (i) and (ii) of \cite[Theorem C]{BBT2025-1}.
\end{rem}

\subsection{Second cohomology}

Here, we illustrate how computations of homology yield cohomological information. For the congruence subgroups under consideration, our description of the first homology determines the torsion subgroup of the second cohomology. We write $H_{\operatorname{tor}}$ for the torsion subgroup of an abelian group $H$. 

\begin{lem}\label{coh-torsion}
    Let $G$ be a group such that $H_n(G, \z)$, $H_{n+1}(G, \z)$ and $H^{n+1}(G, \z)$ are finitely generated. Then
    \[ \operatorname{rank}H_{n +1}(G, \z) = \operatorname{rank} H^{n+1}(G, \z) \text{, and } H_n(G, \z)_{\operatorname{tor}} \simeq H^{n+1}(G, \z)_{\operatorname{tor}}.\]
\end{lem}
\begin{proof}
    See \cite[Lemma 4.3]{mirzaii-perez-2025}.
\end{proof}

\begin{dfn}
    A group $G$ is of type $FP$ (resp. $FL$) if $\z$ admits a finite resolution
    \[ 0 \rightarrow P_{n} \rightarrow P_{n-1} \rightarrow \cdots \rightarrow P_{0} \rightarrow \z \rightarrow 0,\]
    where each $P_i$ is a finitely generated projective (resp. free) $\z[G]$-module.
\end{dfn}
\begin{dfn}
    A group $G$ is of type $FP_{n}$ if it admits a resolution 
    \[ \cdots \rightarrow P_{n} \rightarrow P_{n-1} \rightarrow \cdots \rightarrow P_{0} \rightarrow \z \rightarrow 0,\]
    where each $P_i$ is a finitely generated projective $\z[G]$-module for $i\leq n$. If $G$ is of type $FP_n$, for every $n \geq 0$, we say that $G$ is of type $FP_{\infty}$.
\end{dfn}

\begin{prp} \label{FP-homology}
    If $G$ is a group of type $FP_n$, then $H_k(G, \z)$ and $H^{k}(G, \z)$ are finitely generated, for every $0 \leq k \leq n$.
\end{prp}
\begin{proof}
    See \cite[Exercise 1, \S5, Chap. VIII]{brown1994}.
\end{proof}
%\begin{proof}
    %{\color{red}{(I think this an exercise in Brown's book! Maybe just refer to it.)}} Let
    %\[\cdots \rightarrow P_{n} \rightarrow P_{n-1} \rightarrow \cdots \rightarrow P_{0} \rightarrow \z \rightarrow 0\]
    %be a projective resolution of the $\z[G]$ module $\z$, where each $P_i$ is finitely generated for $i \leq n$. $H_k(G, \z)$ is the homology of the chain complex:
    %\[\cdots \overset{\partial_{n+1}}{\rightarrow} P_{n} \otimes_{\z[G]} \z \overset{\partial_n}{\rightarrow} P_{n-1} \otimes_{\z[G]} \z \overset{\partial_{n-1}}{\rightarrow} \cdots \overset{\partial_1}{\rightarrow} P_{0} \otimes_{\z[G]} \z \arr 0.\]
    %Each $P_i \otimes_{\z[G]} \z$ is a finitely generated $\z$-module (for $i \leq n$) and $\z$ is Noetherian. Since $\ker(\partial_i)$ is a submodule of $P_i \otimes_{\z[G]} \z$, it must be finitely generated. Consequently,
    %\[H_i(G, \z) = \frac{\ker(\partial_i)}{\im(\partial_{i+1})}\]
    %is finitely generated.

    %The result for cohomology is analogous: consider the chain complex obtained by applying $\operatorname{Hom}_{\z[G]}(-, \z)$:
    %\[0 \rightarrow \operatorname{Hom}_{\z[G]}(P_0, \z) \overset{\delta_1}{\rightarrow} \operatorname{Hom}_{\z[G]}(P_1, \z) \overset{\delta_2}{\rightarrow} \cdots \overset{\delta_{n-1}}{\rightarrow} \operatorname{Hom}_{\z[G]}(P_n, \z) \overset{\delta_n}{\rightarrow} \cdots\]
    %One can check that, for $i \leq n$, $\operatorname{Hom}_{\z[G]}(P_i, \z)$ is a finitely generated $\z$-module, so we can proceed as before.
%\end{proof}

\begin{lem} \label{h-fg-nf}
    Suppose $A = \mathcal{O}_{K, S}$, where $K$ is a number field. Then, for every $k \geq 0$, $H_k(\hat{\Gg}(A, \ppp), \z)$ and $H^{k}(\hat{\Gg}(A, \ppp), \z)$ are finitely generated.
\end{lem}
\begin{proof}
    By Selberg's Lemma (see \cite{alperin}), $\hat{\Gg}(A, \ppp)$ contains a torsion-free subgroup $H$, which is of finite-index in $\hat{\Gg}(A, \ppp)$. $H$ is of type $FL$ \cite[Theorem 6.2]{borel-serre1976}, hence it is of type $FP_{\infty}$. Using \cite[Proposition 5.1, Chapter 8]{brown1994} , we conclude that $\hat{\Gg}(A, \ppp)$ is also of type $FP_{\infty}$. The result now follows from Proposition \ref{FP-homology}.
\end{proof}
\begin{lem}\label{h-fg-ff}
    Suppose $A = \mathcal{O}_{K, S}$, where $K$ is a function field. Then, for  $0 \leq k \leq |S| -1$, $H_k(\hat{\Gg}(A, \ppp), \z)$ and $H^{k}(\hat{\Gg}(A, \ppp), \z)$ are finitely generated.
\end{lem}
\begin{proof}

%\begin{rem}
%    $\PGL_2(A)/\PSL_2(A) \simeq A^{\times}/{A^{\times}}^2$, and for us $|A^{\times}|$ is infinite.
%\end{rem}
%\begin{rem}
%Let $A$ be a Dedekind domain of arithmetic type with infinitely many units. It is well-known that $\SL_2(A):=\ker(\det_2)$, where $\det_2: \GL_2(A) \arr \aa$ is the determinant homomorphism. As $A$ is a $\GE_2$-ring, we have that $\Ee_2(A)$ is normal in 
%$\GL_2(A)$. This is a notable fact because there are examples (due to A. Suslin, see 
%\cite{sus1976}) of Dedekind domains for which $\Ee_2(R)$ is not normal in $\GL_2(R)$.
%\end{rem}

    The group $\hat{\Gg}(A, \ppp)$ is commensurable with $\SL_2(\OO_{K,S})$ because it is a subgroup of finite index. This implies, by \cite[Theorem 6]{stuhler1980}, that $\hat{\Gg}(A, \ppp)$ is $FP_{|S| - 1}$.
    %Notice that $\hat{\Gg}(A, \ppp)/\{\pm I\}$ is a subgroup of finite index of $\PGL_2(A)$ 
    %(since $\PSL_2(A)$ has finite index in $\PGL_2(A)$ {\color{red}{(Why? Check this!)}}). Then $\hat{\Gg}(A, \ppp)/\{\pm I\}$ is commensurable with $\PGL_2(A)$, which implies $\hat{\Gg}(A, \ppp)/\{\pm I\}$ is $FP_{|S| - 1}$ by \cite[Theorem 6]{stuhler1980}. Now, there is an exact sequence
    %\[ 0 \rightarrow \{\pm I\} \rightarrow \hat{\Gg}(A, \ppp) \rightarrow \hat{\Gg}(A, \ppp)/\{\pm I\} \rightarrow 0,\]
    %where $\{\pm I\}$ and $\hat{\Gg}(A, \ppp)/\{\pm I\}$ are $FP_{|S| - 1}$. Using \cite[Proposition 2.7]{bieri-hom-dimension}, we conclude that $\hat{\Gg}(A, \ppp)$ is $FP_{|S| - 1}$.
\end{proof}

\begin{rem}
    Stuhler explains in the introduction of \cite{stuhler1980} that the results in his article are stated for $\PGL_2(\OO_{K,S})$ but they also hold for $\GL_2(\OO_{K,S})$ and $\SL_2(\OO_{K,S})$.
\end{rem}

\begin{prp}\label{torsion-prp}
    Assume $A = \mathcal{O}_{K, S}$ is not totally imaginary, with $|S| \geq 2$ if $K$ is a number field and $|S| \geq 3$ if $K$ is a function field. Let $(0) \neq \ppp \in \Spec(A)$. Then, 
    \[H^{2}(\hat{\Gg}(A, \ppp), \z)_{\operatorname{tor}} \simeq H_1(\hat{\Gg}(A, \ppp), \z)_{\operatorname{tor}}.\]
\end{prp}
\begin{proof}
    Follows directly from Lemmas \ref{coh-torsion}, \ref{h-fg-nf}, \ref{h-fg-ff} and Proposition \ref{FP-homology}.
\end{proof}
\begin{cor}
Assume $A = \mathcal{O}_{K, S}$ is not totally imaginary, with $|S| \geq 2$ if $K$ is a number field and $|S| \geq 3$ if $K$ is a function field. If $(0) \neq \ppp \in \Spec(A)$ and $\char(\kappa(\ppp)) \neq 2$, then
    \begin{enumerate}
        \item[(i)] For $|\kappa(\ppp)| \geq 4$, 
        \[H^{2}(\Gg(A, \ppp), \z)_{\operatorname{tor}} \simeq \kappa(\ppp)^3 \oplus H_1(\SL_2(A), \z).\]
        \item[(ii)] For $|\kappa(\ppp)| = 3$, 
        \[H^{2}(\Gg(A, \ppp), \z)_{\operatorname{tor}} \simeq \kappa(\ppp)^2 \oplus H_1(\SL_2(A), \z).\] 
    \end{enumerate}
\end{cor}
\begin{cor}
Assume $A = \mathcal{O}_{K, S}$ is not totally imaginary, with $|S| \geq 2$ if $K$ is a number field and $|S| \geq 3$ if $K$ is a function field. If $(0) \neq \ppp \in \Spec(A)$ and $|\kappa(\ppp)| \geq 3$, then
\[H^{2}(\Gg_1(A, \ppp), \z)_{\operatorname{tor}} \simeq \kappa(\ppp) \oplus H_1(\SL_2(A), \z).\]
\end{cor}
\begin{cor}
    Assume $A = \mathcal{O}_{K, S}$ is not totally imaginary, with $|S| \geq 2$ if $K$ is a number field and $|S| \geq 3$ if $K$ is a function field. If $(0) \neq \ppp \in \Spec(A)$, then
    \begin{enumerate}
        \item[(i)] For $|\kappa(\ppp)| \geq 4$, 
        \[H^{2}(\Gg(A, \ppp), \z)_{\operatorname{tor}} \simeq \kappa(\ppp)^\times \oplus H_1(\SL_2(A), \z).\]
        \item[(ii)] For $|\kappa(\ppp)| = 3$, 
        \[H^{2}(\Gg(A, \ppp), \z)_{\operatorname{tor}} \simeq \kappa(\ppp) \oplus \kappa(\ppp)^{\times} \oplus H_1(\SL_2(A), \z).\] 
    \end{enumerate}
\end{cor}

\end{document}